\documentclass[12pt]{amsart}
\usepackage[centertags]{amsmath}
\usepackage{amstext,amssymb,amsopn,amsthm}
\usepackage{enumerate}
\usepackage{dsfont}
\usepackage[margin=1.2in]{geometry}
\usepackage[english]{babel}
\usepackage{color}

\usepackage{mathtools}

\newcommand\fd{\mathrel{\stackrel{\makebox[0pt]{\mbox{\normalfont\tiny fd}}}{=}}}

\newcommand{\tl}[1]{\textcolor{black}{#1}}

\DeclareMathOperator{\dimH}{dim_H}

\DeclareMathOperator{\Ker}{Ker}

\newcommand{\RR}{\mathbb{R}}
\newcommand{\NN}{\mathbb{N}}

\newcommand{\EE}{\mathbb{E}}

\newtheorem{theorem}{Theorem}
\newtheorem{lemma}[theorem]{Lemma}
\newtheorem{corollary}[theorem]{Corollary}
\newtheorem{proposition}[theorem]{Proposition}

\theoremstyle{definition}

\newtheorem{remark}[theorem]{Remark}

\newenvironment{proof1}{\par\noindent {\textit{Proof of Theorem \ref{th:HdmMultPoints}.}}}
{\begin{flushright} \vspace*{-6mm}\mbox{$\Box$} \end{flushright}}

\newenvironment{proof2a}{\par\noindent {\textit{Proof of Theorem \ref{th:ExMultPoints}, Case (A.1).}}}
{\begin{flushright} \vspace*{-6mm}\mbox{$\Box$} \end{flushright}}

\newenvironment{proof2b}{\par\noindent {\textit{Proof of Theorem \ref{th:ExMultPoints}, Case (A.2).}}}
{\begin{flushright} \vspace*{-6mm}\mbox{$\Box$} \end{flushright}}

\title{Multiple points of operator semistable L\'evy processes}

\subjclass[2010]{60J25, 60J30, 60G51, 60G17}
\keywords{Multiple points, Hausdorff dimension, Operator semistable process, L\'evy process}

\author[T. Luks]{Tomasz Luks}
\address{Institut f\"ur Mathematik, Universit\"at Paderborn, Warburger Strasse 100, D-33098 Paderborn, Germany}
\email{tluks@math.uni-paderborn.de}

\author[Y. Xiao]{Yimin Xiao}
\address{Department of Statistics and Probability, Michigan State University\\
619 Red Cedar Road, C413 Wells Hall,
East Lansing MI 48824-1027, USA}
\email{xiaoyimi@stt.msu.edu}
\thanks{Research of Y. Xiao was partially
supported by the NSF Grants DMS-1612885 and DMS-1607089.}

\begin{document}

\begin{abstract}
We determine the Hausdorff dimension of the set of $k$-multiple points for a symmetric operator semistable 
L\'evy process $X=\{X(t), t\in\RR_+\}$ in terms of the eigenvalues of its stability exponent. We also give a 
necessary and sufficient condition for the existence of $k$-multiple points. Our results extend to all $k\geq2$ 
the recent work \cite{LX}, where the set of double points $(k = 2)$ was studied in the symmetric operator stable 
case.
\end{abstract}

\maketitle

\section{Introduction and statement of the results}\label{sec1}

The questions on the existence of multiple points (or intersections) and the Hausdorff dimension of the set of multiple points of stochastic processes
have been of considerable interest for many years. The problem was originally studied for Brownian motion by Dvoretzky, Erd\"os, Kakutani and 
Taylor \cite{DEK1,DEK2,DEKT}. Their results were later extended to more general L\'evy processes using various techniques, see 
\cite{Dy,E,Frist,H,He,K03,KX1,KX3,LRS,LX,MX,S3,T,X04} and the references therein. Let $X=\left\{X(t),t\in\RR_+\right\}$ be a stochastic process 
with values in $\RR^d$, $d\geq1$, and let  $k\geq2$ be an integer. A point $x\in\RR^d$ is called a {\it $k$-multiple point} of $X$ if there exist 
$k$ distinct times $t_1,\ldots,\, t_k\in\RR_+$ such that
$$
X(t_1)= \ldots =X(t_k)=x.
$$
If $k=2$, then $x$ is also called a {\it double point} of $X$. We denote by $M_k$ the set of $k$-multiple points of $X$. The objective of this 
paper is to study the set $M_k$ for symmetric operator semistable L\'evy processes and to extend to all $k\geq2$ the recent results of \cite{LX}, 
where the set of double points was investigated when $X$ is symmetric operator stable. Recall from \cite{MS,S99} that a L\'evy process $X$ 
is called operator semistable if the distribution $\nu$ of $X(1)$ is full (i.e. not supported on any lower dimensional hyperplane) and there 
exists a linear operator $B$ on $\RR^d$ such that 
\begin{equation}\label{eq:OpSem}
\displaystyle\nu^c =c^B\nu\quad\text{ for some }\, c>1,
\end{equation}
where $\nu^c$ denotes the $c$-fold convolution power of the infinitely divisible law $\nu$ and $c^B\nu(dx)=\nu(c^{-B}dx)$ is the 
image measure of $\nu$ under the linear operator $c^B$. The operator $B$ is called a {\it stability exponent} of $X$. We refer to 
\cite{C,J,LR,Lu} and to the monograph \cite{MS} for more comprehensive information on operator semistable laws. As a consequence 
of \eqref{eq:OpSem}, an operator semistable L\'evy process $X$ is also {\it operator semi-selfsimilar}, i.e., for the constant $c>1$ in 
(\ref{eq:OpSem}),
\begin{equation}\label{eq:OpSelf}
\left\{X(ct)\right\}_{t\geq0}\fd\left\{c^BX(t)\right\}_{t\geq0},
\end{equation}
where $\fd$ denotes equality of all finite-dimensional distributions of the process. Operator semi-selfsimilar processes 
constitute a much broader class than that of semi-selfsimilar processes, see \cite{MaSa} for more details. If \eqref{eq:OpSem} 
(resp. \eqref{eq:OpSelf}) holds for all $c > 0$, the L\'evy process $X$ is called {\it operator stable} (resp. {\it operator self-similar}).

To formulate our main results, we factor the minimal polynomial of the stability exponent $B$ into $q_1(x)\cdots q_p(x)$, where 
all roots of $q_i(x)$ have real parts $a_i$ and $a_i<a_j$ for $i<j$. Define $V_i=\Ker(q_i(B))$ and $d_i=\dim(V_i)$. Then $d_1+ 
\cdots +d_p=d$ and $V_1\oplus\cdots\oplus V_p$ is a direct sum decomposition of $\RR^d$ into $B$-invariant subspaces. We 
may write $B=B_1\oplus\cdots\oplus B_p$, where $B_i\colon V_i\to V_i$ and every eigenvalue of $B_i$ has real part equal to 
$a_i$. For $j=1,...,d$ and $l=1,...,p$, denote $\alpha_j=a_l^{-1}$ whenever $\sum_{i=0}^{l-1}d_i<j \leq\sum_{i=0}^{l}d_i$, where 
$d_0:=0$. We then have $\alpha_1\geq...\geq\alpha_d$, and note that $0<\alpha_j\leq2$ in view of \cite[Theorem 7.2.1]{MS}. 

Our first theorem provides an explicit formula for the Hausdorff dimension of $M_k$ in $\RR^2$ in terms of the indices $\alpha_j$.

\begin{theorem}\label{th:HdmMultPoints}
Let $X =\left\{X(t),t\in\RR_+\right\}$ be a symmetric operator semistable L\'evy process
in $\RR^2$ with exponent $B$ and let $M_k$ be the set of $k$-multiple points of $X$. Then for all $k\geq2$ we have almost surely
$$
\dimH M_k=\min\left\{\alpha_1\left(k-(k-1)\left(\alpha^{-1}_1+\alpha^{-1}_2\right)\right), \,
2-k\alpha_2\left(\alpha^{-1}_1+\alpha^{-1}_2-1\right)\right\},
$$
where a negative dimension means that $M_k=\emptyset$ \tl{almost surely}.
\end{theorem}

 Theorem~\ref{th:HdmMultPoints} is more general than \cite[Theorem 1]{S3}, where $B$ is assumed to be a diagonal matrix 
with entries on the diagonal $\alpha_j \in (1, 2)$ ($1 \le j \le d$). Note also that the dimension formula for double points in \tl{$\RR^2$ 
and} $\RR^3$ is given in \cite[Corollary 3.8]{KMX}. Since $M_3=\emptyset$ a.s. for $d=3$ and $M_2=\emptyset$ a.s. for $d\geq4$ 
(see the beginning of Section~\ref{sec4} for the proof), Theorem~\ref{th:HdmMultPoints} completes the solution of the Hausdorff 
dimension problem for $M_k$ in the setting of symmetric operator semistable L\'evy processes.

Our second theorem characterizes the existence of multiple points in terms of $\alpha_1$ and $\alpha_2$. According to the Jordan 
decomposition for $d=2$, the stability exponent of $X$ satisfies $B=PDP^{-1}$ for some real invertible matrix $P$ and a matrix $D$ 
which can have the following forms:
\tl{\begin{enumerate}
\item[(A.1)] $\left(
	\begin{matrix}
		1/\alpha_1 & 0 \\
		0 & 1/\alpha_2  \\
	\end{matrix}
	\right) $ \ or \ $\left(
	\begin{matrix}
		1/\alpha_1 & -b \\
		b & 1/\alpha_2  \\
	\end{matrix}
	\right)\ \hbox{ with }\, \alpha_1=\alpha_2;$
\item[(A.2)] $\left(
	\begin{matrix}
		1/\alpha & 0 \\
		1 & 1/\alpha  \\
	\end{matrix}
	\right)$.
\end{enumerate}
Clearly, in the case (A.2) we have $\alpha:=\alpha_1=\alpha_2$.
}

\begin{theorem}\label{th:ExMultPoints}
Let $X =\left\{X(t),t\in\RR_+\right\}$ be a symmetric operator semistable L\'evy process
in $\RR^d$ and let $k\geq3$. The existence of $k$-multiple points of $X$ for $d=2$ 
depends on the cases (A.1) and (A.2) as follows:
\begin{itemize}
\item In Case (A.1), $M_k \ne \emptyset$ a.s.  if and only if $k-(k-1)(\alpha^{-1}_1+\alpha^{-1}_2)>0$.
\item In Case (A.2), $M_k \ne \emptyset$ a.s.  if and only if $\alpha\geq 2(k-1)/k$.
\end{itemize}
Furthermore, $M_k=\emptyset$ a.s. for $d\geq3$.
\end{theorem}

The existence of double points is characterized separately in Corollary~\ref{cor:ExDouble}. 
Note that in the case (A.2) and $k\geq3$, the set $M_k$ is nonempty almost surely if $\alpha=2(k-1)/k$ 
although its Hausdorff dimension is zero. The same effect appears for double points in $\RR^3$ in 
the non-diagonalizable case (B.3) and $\alpha=3/2$ (see the end of Section~\ref{sec4} for details). 

\tl{Even though in this paper we focus on symmetric operator semistable L\'evy process in $\RR^d$ 
with $d \ge 2$ so that its exponent $B$ is a matrix, the problems on existence of $k$-multiple points
and the Hausdorff dimension of $M_k$ are meaningful for the case of $d=1$ as well. In the special case 
when $X= \left\{X(t),t\in\RR_+\right\}$ is a symmetric stable L\'evy process in $\RR$, these problems 
were solved by Taylor \cite{T}. For a symmetric semistable L\'evy process in $\RR$, we can 
apply the general tools in \cite{KX4, KX5,LRS,LX}  (cf. (\ref{eq:LevExpCond}) and (\ref{eq:HdmGeneral}) 
below) to solve these problems. We provide more details in Remark \ref{Re:last}. 
}
 
This paper raises several natural questions on multiple points of operator semistable L\'evy processes. 
For example, it would be interesting to remove the ``symmetry'' assumption; or to find an exact 
Hausdorff measure function for $M_k$. Another interesting problem is to find the packing dimension 
of $M_k$. 

Our paper is organized as follows. In Section~\ref{sec2} we give basic definitions and recall some useful 
facts.  In Section~\ref{sec3} we prove Theorem~\ref{th:HdmMultPoints}. In Section~\ref{sec4} we deal 
with the existence problem for multiple points and prove Theorem~\ref{th:ExMultPoints}.

\section{Preliminaries}\label{sec2}

A stochastic process $X=\left\{X(t),t\in\RR_+\right\}$ with values in $\RR^d$ is called a L\'evy process if $X$ has 
stationary and independent increments, $X(0)=0$ a.s. and $t\mapsto X(t)$ is continuous in probability. We refer 
to the books \cite{B96,S99} for systematic accounts on L\'evy processes. It is known that the finite-dimensional 
distributions of $X$ are determined by the characteristic function
$$
\EE[ e^{i \langle \xi,\, X(t)\rangle }]= e^{-t\Psi(\xi)}, \quad
\forall\, t\geq0,
$$
where $\Psi\colon\RR^d\mapsto\mathbb {C}$ is given by the
L\'evy-Khintchine formula and is called the {\it characteristic or L\'evy exponent} of $X$.

A L\'evy process $X$ is said to be {\it symmetric} if $-X$ and $X$ have the same finite-dimensional distributions.
In such a case,  $\Psi (\xi) \ge 0$ for all $\xi \in \RR^d$. Using the terminology in \cite{KX4,KX5}, we say that $X$ 
is {\it absolutely continuous}, if for all $t>0$, the function $\xi\mapsto e^{-t\Psi(\xi)}$ is in $L^1(\RR^d)$. 
In this case, the Fourier inversion formula implies that the density function of $X(t)$ is bounded and continuous.

 It has been recently proved in \cite[Theorem 1]{LX} that the Hausdorff dimension of $M_k$ for a symmetric, 
 absolutely continuous L\'evy process with characteristic exponent $\Psi$ is given by
\begin{equation}\label{eq:HdmGeneral}
\dimH M_k= d-\inf\left\{\beta\in(0,d]:\int_{\RR^{kd}}\left[\frac{1}{1+\|\sum_{l=1}^k\xi_l\|^{\beta}}
\prod_{j=1}^k\frac{1}{1+\Psi(\xi_j)} \right]d\overline{\xi}<\infty \right\}
\end{equation}
almost surely, where $\overline{\xi}=(\xi_1,...,\xi_k)$ for $\xi_j\in\RR^d$ \tl{and $\|\cdot\|$ denotes the usual Euclidean 
norm in $\RR^d$}. Here we use the convention $\inf\emptyset=d$. 

Furthermore, in the case when $X$ is a symmetric operator semistable L\'evy process, \cite[Corollary 2.2]{KMX} gives 
the following estimate for its characteristic exponent $\Psi$: for every $\varepsilon>0$, there exists a constant $\tau>1$ 
such that for all $\xi\in\RR^d$ with $\|\xi\|\geq\tau$, we have
\begin{equation}\label{est:Levyexp}
\frac{K^{-1}}{\|\xi\|^{\varepsilon}\sum_{j=1}^d|\xi_j|^{\alpha_j}}\leq
\frac{1}{1+\Psi(\xi)}\leq\frac{\tl{K}}{\sum_{j=1}^d|\xi_j|^{\alpha_j}},
\end{equation}
where $K\geq1$ is a constant which depends on $\varepsilon$ and $\tau$ only and, as defined in the Introduction,  
\tl{$\alpha_1\geq ... \geq\alpha_d >0$ are the reciprocals of the} 
real parts of the eigenvalues of the stability exponent $B$. The estimate \eqref{est:Levyexp} 
extends \cite[Theorem 4.2]{MX}, where the result was obtained for operator stable L\'evy processes, and 
implies that $X$ is absolutely continuous.

Throughout the rest of the paper, $C$ will denote a positive constant, whose value may change in each appearance 
and depends on the coefficients $\alpha_j$ and the multiplicity $k$ only. By the notation $f\asymp g$ we mean that 
there is a constant $C$ such that $C^{-1}g\leq f\leq Cg$. Analogously we define $\gtrsim$ and $\lesssim$.

\section{Hausdorff dimension of multiple points in $\RR^2$}\label{sec3}

Our goal in this section is to  prove Theorem~\ref{th:HdmMultPoints} by extending the method
of \cite{LX} in order to deal with the set $M_k$ for $k \ge 3$. To this end we fix $2\geq\alpha_1\geq\alpha_2>0$ 
and let $X =\left\{X(t),t\in\RR_+\right\}$ be a symmetric operator semistable L\'evy process in $\RR^2$ with 
stability exponent $B$ whose eigenvalues have real parts $\alpha^{-1}_1,\alpha^{-1}_2$, as explained in the Introduction. 
Denote
$$
A_k:=\left\{(x_1,...,x_k)\in\RR^{2k}:\tl{\|x_i\|}>1, i=1,...,k\right\},
$$
where $x_i=(x_{i1},x_{i2})\in\RR^2$ for $i=1,...,k$, and for $\beta>0$ let
\begin{align*}
I_{\beta}:=\idotsint_{A_k} \frac{1}{1+|\sum_{i=1}^kx_{i1}|^{\beta}+|\sum_{i=1}^kx_{i2}|^{\beta}}
\prod_{i=1}^k\frac{dx_1 \dots dx_k}{|x_{i1}|^{\alpha_1}+|x_{i2}|^{\alpha_2}}.
\end{align*}
It follows from  \eqref{eq:HdmGeneral},  \eqref{est:Levyexp} and a slight modification of 
\cite[Lemma 2]{LX} that   
\begin{equation}\label{eq:DimTriple}
\dimH M_k=2-\inf\{\beta\in(0,2):I_\beta<\infty\}
\end{equation}
almost surely. For $q,r\geq1$ let
\begin{align*}
A_k(q,r):=&\biggl\{(x_1,...,x_k)\in\RR^{2k}:\tl{\|x_i\|}>1,i=1,...,k,\biggl.\\
&\qquad \biggl. q-1\leq \Big|\sum_{i=1}^kx_{i1} \Big|<q, \ r-1\leq \Big|\sum_{i=1}^kx_{i2} \Big|<r\biggl\}.
\end{align*}
We have
\begin{equation}\label{eq:2series}
I_{\beta}<\infty\ \ \text{ if and only if }\ \ \sum_{m,n\in\NN}\frac{1}{m^{\beta}+n^{\beta}}\idotsint_{A_k(m,n)}
\prod_{i=1}^k\frac{dx_1 \dots dx_k}{|x_{i1}|^{\alpha_1}+|x_{i2}|^{\alpha_2}}<\infty.
\end{equation}
In \cite[Proof of Theorem 4]{LX} it has been shown that, excepting the cases $\alpha_1=\alpha_2=2$ and $\alpha_1=2$, $\alpha_2=1$, 
one has
\begin{equation}\label{DPintegral}
\iint_{A_2(m,n)}\frac{dx_1dx_2}{(|x_{11}|^{\alpha_1}+|x_{12}|^{\alpha_2})(|x_{21}|^{\alpha_1}+|x_{22}|^{\alpha_2})}
\asymp(m^{\alpha_1}+n^{\alpha_2})^{1/\alpha_1+1/\alpha_2-2}
\end{equation}
for any $m,n\in\NN$ provided $2-1/\alpha_1-1/\alpha_2>0$ and the integral above is infinite otherwise. When $\alpha_1=\alpha_2=2$ 
or $\alpha_1=2$ and $\alpha_2=1$, an extra factor $\log m$ or $\log n$ may appear in the upper bound of \eqref{DPintegral}, 
the rest of the statement remains the same. Our next result extends \eqref{DPintegral} to all $k\geq2$, which is essential for proving 
Theorem \ref{th:HdmMultPoints}.

\begin{proposition}\label{th:Ibeta}
Assume $\alpha_2<2$ and let $k\geq2$ be an integer. For $k=2=\alpha_1$ assume in addition $\alpha_2\neq1$. 
Then for any real numbers $q,r\geq1$ we have
\begin{align}\label{TRintegral}
\idotsint_{A_k(q,r)}\prod_{i=1}^k\frac{1}{|x_{i1}|^{\alpha_1}+|x_{i2}|^{\alpha_2}}\,dx_1 \dots dx_k\asymp(q^{\alpha_1}
+r^{\alpha_2})^{(k-1)(1/\alpha_1+1/\alpha_2)-k}
\end{align}
provided $k-(k-1)(1/\alpha_1+1/\alpha_2)>0$, and the integral is infinite otherwise.
\end{proposition}

\begin{proof}
Notice that the integrand in \eqref{TRintegral} is symmetric in $x_{i1}$ and $x_{i2}$ for $i=1,...,k$. Hence, without loss of generality, 
we  will consider the integral  in \eqref{TRintegral} over the set
\begin{align*}
\widetilde A_k(q,r):=&\biggl\{(x_1,...,x_k)\in\RR^{2k}:x_{k1}, x_{k2}\geq1, |x_{i1}|, |x_{i2}|\geq 1,i=1,...,k-1,\Bigl.\\
&\qquad \Bigl.q-1\leq \Big|\sum_{i=1}^kx_{i1} \Big|<q, \ r-1\leq \Big|\sum_{i=1}^kx_{i2} \Big|<r\biggl\}.
\end{align*}
We then have $q-1\leq|\sum_{i=1}^kx_{i1}|<q$ if and only if
$$
-q-x_{k1}\leq \sum_{i=1}^{k-1}x_{i1}\leq -q-x_{k1}+1\quad\text{ or }\quad q-x_{k1}-1\leq \sum_{i=1}^{k-1}x_{i1}\leq q-x_{k1},
$$
and $r-1\leq|\sum_{i=1}^kx_{i2}|<r$ if and only if
$$
-r-x_{k2}\leq \sum_{i=1}^{k-1}x_{i2}\leq -r-x_{k2}+1\quad\text{ or }\quad r-x_{k2}-1\leq \sum_{i=1}^{k-1}x_{i2}\leq r-x_{k2}.
$$
Consider the following four cases:
\begin{enumerate}
	\item[1).] $q+x_{k1}-1\leq |\sum_{i=1}^{k-1}x_{i1}|\leq q+x_{k1}$ and $x_{k1}\geq 1$.
	\item[2).] $q-x_{k1}-1\leq |\sum_{i=1}^{k-1}x_{i1}|\leq q-x_{k1}$ and $1\leq x_{k1}\leq q-1$.
	\item[3).] $|\sum_{i=1}^{k-1}x_{i1}|\leq 2$ and $q-1\leq x_{k1}\leq q+1$.
	\item[4).] $x_{k1}-q\leq |\sum_{i=1}^{k-1}x_{i1}|\leq x_{k1}-q+1$ and $x_{k1}\geq q+1$.
\end{enumerate}
It can be seen that  the condition $q-1\leq|\sum_{i=1}^{k}x_{i1}|<q$ implies one of the four cases above. Consider analogous 
cases implied by $r-1\leq|\sum_{i=1}^{k}x_{i2}|<r$, and for $i,j=1,2,3,4$ let $A_k^{i,j}(q,r)$ denote the subset of
$$
\{(x_1,...,x_k)\in\RR^{2k}:x_{k1}, x_{k2}\geq1, |x_{i1}|, |x_{i2}|\geq 1,i=1,...,k-1\}
$$
with the case $i$ applied to $x_{11},...,x_{k1}$ and the case $j$ applied to $x_{12},...,x_{k2}$. For instance,
\begin{align*}
A_k^{1,1}(q,r)&=\Bigl\{(x_1,...,x_k)\in\RR^{2k}:x_{k1}, x_{k2}\geq1, |x_{i1}|, |x_{i2}|\geq 1,i=1,...,k-1,\Bigl.\\
&\Bigl.q+x_{k1}-1\leq\Big|\sum_{i=1}^{k-1}x_{i1}\Big|\leq q+x_{k1}, \ \ r+x_{k2}-1\leq\Big|\sum_{i=1}^{k-1}x_{i2}\Big|<r+x_{k2}\Bigl\}.
\end{align*}
We have
$$
A_k^{1,1}(q,r)\subseteq\widetilde A_k(q,r)\subseteq \bigcup_{i,j=1}^4 A_k^{i,j}(q,r).
$$
Hence, it is enough to show the lower bound of \eqref{TRintegral} with the integration restricted to $A_k^{1,1}(q,r)$ 
and the upper bound of \eqref{TRintegral} with the integration over all $A_k^{i,j}(q,r)$.

We proceed by induction on $k$. Note that the proposition is proved to hold for $k=2$ except the upper 
bound in the case $\alpha_1=2$ and $\alpha_2=1$. Clearly, analyzing \cite[Proof of Theorem 4]{LX} one can easily 
deduce that \eqref{DPintegral} holds also for any non-integer numbers $m,n\geq 1$.

Assume now the statement of 
the proposition holds for some $k\geq2$ and let $k'=k+1$. Consider first the integration over $A_{k'}^{1,1}(q,r)$. By 
applying the induction hypothesis  to $A_{k}(q+x_{k'1},r+x_{k'2})$ we get
\begin{equation}\label{Eq:star}
\begin{split}
&\idotsint_{A^{1,1}_{k'}(q,r)}\prod_{i=1}^{k'}\frac{1}{|x_{i1}|^{\alpha_1}+|x_{i2}|^{\alpha_2}}\,dx_1 \dots dx_{k'}\\
&\asymp\int_1^\infty\int_1^\infty\left(\frac{1}{(q+x_{k'1})^{\alpha_1}+(r+x_{k'2})^{\alpha_2}}\right)^{k-(k-1)(1/\alpha_1+1/\alpha_2)}\frac{dx_{k'1}dx_{k'2}}{x_{k'1}^{\alpha_1}+x_{k'2}^{\alpha_2}}\\
&\asymp\int_1^\infty\int_1^\infty\left(\frac{1}{q^{\alpha_1}+r^{\alpha_2}+t^{\alpha_1}+s^{\alpha_2}}\right)^{k-(k-1)(1/\alpha_1+1/\alpha_2)}\frac{dtds}{t^{\alpha_1}+s^{\alpha_2}}\\
&\asymp\int_1^\infty\int_1^\infty\left(\frac{1}{q^{\alpha_1}+r^{\alpha_2}}\wedge\frac{1}{t^{\alpha_1}}\wedge\frac{1}{s^{\alpha_2}}\right)^{k-(k-1)(1/\alpha_1+1/\alpha_2)}\left(\frac{1}{t^{\alpha_1}}\wedge\frac{1}{s^{\alpha_2}}\right)dtds\\
&=\int_1^\infty\int_1^{s^{\alpha_2/\alpha_1}}\frac{1}{s^{\alpha_2}}\left(\frac{1}{q^{\alpha_1}+r^{\alpha_2}}\wedge\frac{1}{s^{\alpha_2}}\right)^{k-(k-1)(1/\alpha_1+1/\alpha_2)}dtds\\
&\qquad +\int_1^\infty\int_{s^{\alpha_2/\alpha_1}}^\infty\frac{1}{t^{\alpha_1}}\left(\frac{1}{q^{\alpha_1}+r^{\alpha_2}}\wedge\frac{1}{t^{\alpha_1}}\right)^{k-(k-1)(1/\alpha_1+1/\alpha_2)}dtds\\
& := I_1+I_2.
\end{split}
\end{equation}
Recall that for $k=2=\alpha_1$ and $\alpha_2=1$ only the lower bound of the above is true. We have
\begin{equation}\label{Eq:star2}
\begin{split}
I_2&=\int_1^\infty\int_{s^{\alpha_2/\alpha_1}}^\infty\frac{1}{t^{\alpha_1}}\left(\frac{1}{q^{\alpha_1}+r^{\alpha_2}}\wedge\frac{1}
{t^{\alpha_1}}\right)^{k-(k-1)(1/\alpha_1+1/\alpha_2)}dtds\\
&=\int_1^{(q^{\alpha_1}+r^{\alpha_2})^{1/\alpha_2}}\int_{s^{\alpha_2/\alpha_1}}^\infty\frac{1}{t^{\alpha_1}}\left(\frac{1}
{q^{\alpha_1}+r^{\alpha_2}}\wedge\frac{1}{t^{\alpha_1}}\right)^{k-(k-1)(1/\alpha_1+1/\alpha_2)}dtds\\
&\qquad  +\int^\infty_{(q^{\alpha_1}+r^{\alpha_2})^{1/\alpha_2}}\int_{s^{\alpha_2/\alpha_1}}^\infty t^{(k-1)(1+\alpha_1/\alpha_2)-k'\alpha_1}dtds
:=I_2^{(1)}+I_2^{(2)}.
\end{split}
\end{equation}
For $k'-k(1/\alpha_1+1/\alpha_2)>0$ we get
\begin{align*}
I_2^{(2)}=C\int^\infty_{(q^{\alpha_1}+r^{\alpha_2})^{1/\alpha_2}} s^{k(\alpha_2/\alpha_1+1)-k'\alpha_2-1}ds=C(q^{\alpha_1}+r^{\alpha_2})^{k(1/\alpha_1+1/\alpha_2)-k'}.
\end{align*}
This gives the lower bound in \eqref{TRintegral}. If $k'-k(1/\alpha_1+1/\alpha_2)\leq0$, then $I_2^{(2)}=\infty$, so the same holds for 
the integral in \eqref{TRintegral} and the last part of the statement follows. 

It remains to prove the upper bound in \eqref{TRintegral}  for the case of $k'-k(1/\alpha_1+1/\alpha_2)>0$. Note that the latter also implies 
$\alpha_2>1$, so the restriction in the case $k=2$ is no longer needed. We now proceed to bound the integrals in \eqref{TRintegral} over 
the sets $A^{i,j}_{k'}(q,r)$, separately. 

On the set $A^{1,1}_{k'}(q,r)$, we make use of (\ref{Eq:star}) and (\ref{Eq:star2})  and derive
\begin{align*}
I_2^{(1)}&=\int_1^{(q^{\alpha_1}+r^{\alpha_2})^{1/\alpha_2}}\int_{s^{\alpha_2/\alpha_1}}^\infty\frac{1}{t^{\alpha_1}}
\left(\frac{1}{q^{\alpha_1}
+r^{\alpha_2}}\wedge\frac{1}{t^{\alpha_1}}\right)^{k-(k-1)(1/\alpha_1+1/\alpha_2)}dtds\\
&=\left(\frac{1}{q^{\alpha_1}+r^{\alpha_2}}\right)^{k-(k-1)(1/\alpha_1+1/\alpha_2)}\int_1^{(q^{\alpha_1}+r^{\alpha_2})^{1/\alpha_2}}
\int_{s^{\alpha_2/\alpha_1}}^{(q^{\alpha_1}+r^{\alpha_2})^{1/\alpha_1}}t^{-\alpha_1}dtds\\
&+\int_1^{(q^{\alpha_1}+r^{\alpha_2})^{1/\alpha_2}}\int_{(q^{\alpha_1}+r^{\alpha_2})^{1/\alpha_1}}^\infty t^{(k-1)(1+\alpha_1/\alpha_2)-k'\alpha_1}dtds\\
&\lesssim\left(\frac{1}{q^{\alpha_1}+r^{\alpha_2}}\right)^{k-(k-1)(1/\alpha_1+1/\alpha_2)}\int_1^{(q^{\alpha_1}+r^{\alpha_2})^{1/\alpha_2}}
s^{\alpha_2/\alpha_1-\alpha_2}ds\\
&+\int_1^{(q^{\alpha_1}+r^{\alpha_2})^{1/\alpha_2}} (q^{\alpha_1}+r^{\alpha_2})^{(k-1)(1/\alpha_1+1/\alpha_2)-k'+1/\alpha_1}ds\\
&\lesssim(q^{\alpha_1}+r^{\alpha_2})^{k(1/\alpha_1+1/\alpha_2)-k'},
\end{align*}
where the last inequality follows from the assumption $\alpha_2<2$. Furthermore,
\begin{align*}
I_1&=\int_1^\infty\int_1^{s^{\alpha_2/\alpha_1}}\frac{1}{s^{\alpha_2}}\left(\frac{1}{q^{\alpha_1}+r^{\alpha_2}}\wedge\frac{1}{s^{\alpha_2}}\right)^{k-(k-1)(1/\alpha_1+1/\alpha_2)}dtds\\
&\leq\int_1^\infty s^{\alpha_2/\alpha_1-\alpha_2}\left(\frac{1}{q^{\alpha_1}+r^{\alpha_2}}\wedge\frac{1}{s^{\alpha_2}}\right)^{k-(k-1)(1/\alpha_1+1/\alpha_2)}ds\\
&=\left(\frac{1}{q^{\alpha_1}+r^{\alpha_2}}\right)^{k-(k-1)(1/\alpha_1+1/\alpha_2)}\int_1^{(q^{\alpha_1}+r^{\alpha_2})^{1/\alpha_2}} s^{\alpha_2/\alpha_1-\alpha_2}ds\\
&+\int_{(q^{\alpha_1}+r^{\alpha_2})^{1/\alpha_2}}^\infty s^{k(\alpha_2/\alpha_1+1)-k'\alpha_2-1}ds\\
&\lesssim(q^{\alpha_1}+r^{\alpha_2})^{k(1/\alpha_1+1/\alpha_2)-k'}.
\end{align*}
This proves the desired upper bound for the integral in \eqref{TRintegral}  over $A_{k'}^{1,1}(q,r)$.

Next we consider the integral over $A_{k'}^{2,2}(q,r)$. By applying the induction hypothesis to the integral over $A_{k}(q-x_{k'1},r-x_{k'2})$ we get
\begin{align*}
&\idotsint_{A^{2,2}_{k'}(q,r)}\prod_{i=1}^{k'}\frac{1}{|x_{i1}|^{\alpha_1}+|x_{i2}|^{\alpha_2}}\,dx_1 \dots dx_{k'}\\
&\asymp\int_1^{r-1}\int_1^{q-1}\left(\frac{1}{(q-x_{k'1})^{\alpha_1}+(r-x_{k'2})^{\alpha_2}}\right)^{k-(k-1)(1/\alpha_1+1/\alpha_2)}\frac{dx_{k'1}dx_{k'2}}{x_{k'1}^{\alpha_1}+x_{k'2}^{\alpha_2}}\\
&\leq (q^{\alpha_1}+r^{\alpha_2})^{(k-1)(1/\alpha_1+1/\alpha_2-1)}\int_1^{r-1}\int_1^{q-1}\frac{dtds}{((q-t)^{\alpha_1}+(r-s)^{\alpha_2})(t^{\alpha_1}+s^{\alpha_2})}.
\end{align*}
The last double integral above has already appeared in \cite[(11)]{LX}, so it follows from \cite[Proof of Theorem 4]{LX} that
$$
\int_1^{r-1}\int_1^{q-1}\frac{dtds}{((q-t)^{\alpha_1}+(r-s)^{\alpha_2})(t^{\alpha_1}+s^{\alpha_2})}\leq C(q^{\alpha_1}+r^{\alpha_2})^{1/\alpha_1+1/\alpha_2-2}.
$$
This gives the desired upper bound for the integral over $A_{k'}^{2,2}(q,r)$.

For the integral over $A^{3,3}_{k'}(q,r)$, we have
\begin{align*}
&\idotsint_{A^{3,3}_{k'}(q,r)}\prod_{i=1}^{k'}\frac{1}{|x_{i1}|^{\alpha_1}+|x_{i2}|^{\alpha_2}}\,dx_1 \dots dx_{k'}\\
&\lesssim\frac{1}{q^{\alpha_1}+r^{\alpha_2}}\idotsint_{B_k}\prod_{i=1}^{k}\frac{1}{|x_{i1}|^{\alpha_1}+|x_{i2}|^{\alpha_2}}\,dx_1 \dots dx_{k},
\end{align*}
where $B_k=A_{k}(1,1)\cup A_{k}(2,2)\cup A_{k}(2,1)\cup A_{k}(1,2)$. By the induction hypothesis, the last multiple integral 
is finite provided $k-(k-1)(1/\alpha_1+1/\alpha_2)>0$. Since $1/\alpha_1+1/\alpha_2\geq1$, we have
$$
k-(k-1)(1/\alpha_1+1/\alpha_2)\geq k'-k(1/\alpha_1+1/\alpha_2)>0.
$$
Furthermore, $(q^{\alpha_1}+r^{\alpha_2})^{-1}\leq(q^{\alpha_1}+r^{\alpha_2})^{k(1/\alpha_1+1/\alpha_2)-k'}$. This gives 
the upper bound in \eqref{TRintegral} with the integration over $A^{3,3}_{k'}(q,r)$.

Consider $A_{k'}^{4,4}(q,r)$. We apply the induction hypothesis to the integration over $A_{k}(x_{k'1}-q,x_{k'2}-r)$ and get
\begin{align*}
&\idotsint_{A^{4,4}_{k'}(q,r)}\prod_{i=1}^{k'}\frac{1}{|x_{i1}|^{\alpha_1}+|x_{i2}|^{\alpha_2}}\,dx_1 \dots dx_{k'}\\
&\asymp\int_{r+1}^\infty\int_{q+1}^\infty\left(\frac{1}{(x_{k'1}-q)^{\alpha_1}+(x_{k'2}-r)^{\alpha_2}}\right)^{k-(k-1)(1/\alpha_1+1/\alpha_2)}
\frac{dx_{k'1}dx_{k'2}}{x_{k'1}^{\alpha_1}+x_{k'2}^{\alpha_2}}.
\end{align*}
After the change of variables $t=x_{k'1}-q$ and $s=x_{k'2}-r$ the last term equals
\begin{align*}
&\int_{1}^\infty\int_{1}^\infty\left(\frac{1}{t^{\alpha_1}+s^{\alpha_2}}\right)^{k-(k-1)(1/\alpha_1+1/\alpha_2)}\frac{dtds}{(t+q)^{\alpha_1}+(s+r)^{\alpha_2}}\\
&\asymp\int_1^\infty\int_1^\infty\left(\frac{1}{t^{\alpha_1}+s^{\alpha_2}}\right)^{k-(k-1)(1/\alpha_1+1/\alpha_2)}\frac{dtds}{q^{\alpha_1}+r^{\alpha_2}+t^{\alpha_1}+s^{\alpha_2}}\\
&\asymp\int_1^\infty\int_1^\infty\left(\frac{1}{t^{\alpha_1}}\wedge\frac{1}{s^{\alpha_2}}\right)^{k-(k-1)(1/\alpha_1+1/\alpha_2)}\left(\frac{1}{q^{\alpha_1}+r^{\alpha_2}}\wedge\frac{1}{t^{\alpha_1}}\wedge\frac{1}{s^{\alpha_2}}\right)dtds\\
&=\int_1^\infty\int_1^{s^{\alpha_2/\alpha_1}}s^{(k-1)(1+\alpha_2/\alpha_1)-k\alpha_2}\left(\frac{1}{q^{\alpha_1}+r^{\alpha_2}}\wedge\frac{1}{s^{\alpha_2}}\right)dtds\\
&+\int_1^\infty\int_{s^{\alpha_2/\alpha_1}}^\infty t^{(k-1)(1+\alpha_1/\alpha_2)-k\alpha_1}\left(\frac{1}{q^{\alpha_1}+r^{\alpha_2}}\wedge
\frac{1}{t^{\alpha_1}}\right)dtds=\widetilde I_1+\widetilde I_2.
\end{align*}
We have
\begin{align*}
\widetilde I_1&\leq \int_1^\infty s^{k(1+\alpha_2/\alpha_1-\alpha_2)-1}\left(\frac{1}{q^{\alpha_1}+r^{\alpha_2}}\wedge\frac{1}
{s^{\alpha_2}}\right)ds\\
&=\frac{1}{q^{\alpha_1}+r^{\alpha_2}}\int_1^{(q^{\alpha_1}+r^{\alpha_2})^{1/\alpha_2}} s^{k(1+\alpha_2/\alpha_1-\alpha_2)-1}ds
+\int_{(q^{\alpha_1}+r^{\alpha_2})^{1/\alpha_2}}^\infty s^{k(1+\alpha_2/\alpha_1)-k'\alpha_2-1}ds\\
&\lesssim(q^{\alpha_1}+r^{\alpha_2})^{k(1/\alpha_1+1/\alpha_2)-k'}.
\end{align*}
In the last step above we have used again the assumption $\alpha_2<2$. Furthermore,
\begin{align*}
\widetilde I_2&=\frac{1}{q^{\alpha_1}+r^{\alpha_2}}\int_1^{(q^{\alpha_1}+r^{\alpha_2})^{1/\alpha_2}}
\int_{s^{\alpha_2/\alpha_1}}^{(q^{\alpha_1}+r^{\alpha_2})^{1/\alpha_1}} t^{(k-1)(1+\alpha_1/\alpha_2)-k\alpha_1}dtds\\
&\qquad +\int_1^{(q^{\alpha_1}+r^{\alpha_2})^{1/\alpha_2}}\int_{(q^{\alpha_1}+r^{\alpha_2})^{1/\alpha_1}}^\infty t^{(k-1)(1+\alpha_1/\alpha_2)-k'\alpha_1}dtds\\
&\qquad +\int_{(q^{\alpha_1}+r^{\alpha_2})^{1/\alpha_2}}^\infty\int_{s^{\alpha_2/\alpha_1}}^\infty t^{(k-1)(1+\alpha_1/\alpha_2)-k'\alpha_1}dtds.
\end{align*}
The last two integrals appeared already in the case of $A_{k'}^{1,1}(q,r)$, i.e., $I_2^{(2)}$ and the second part of $I_2^{(1)}$, thus they are less than $C(q^{\alpha_1}+r^{\alpha_2})^{k(1/\alpha_1+1/\alpha_2)-k'}$. In order to estimate the remaining term we note again that ${1/\alpha_1+1/\alpha_2}\geq1$, hence
\begin{align*}
&\frac{1}{q^{\alpha_1}+r^{\alpha_2}}\int_1^{(q^{\alpha_1}+r^{\alpha_2})^{1/\alpha_2}}\int_{s^{\alpha_2/\alpha_1}}^{(q^{\alpha_1}
+r^{\alpha_2})^{1/\alpha_1}} t^{-\alpha_1}(t^{\alpha_1})^{(k-1)(1/\alpha_1+1/\alpha_2-1)}dtds\\
&\le (q^{\alpha_1}+r^{\alpha_2})^{(k-1)(1/\alpha_1+1/\alpha_2)-k}\int_1^{(q^{\alpha_1}+r^{\alpha_2})^{1/\alpha_2}}
\int_{s^{\alpha_2/\alpha_1}}^{(q^{\alpha_1}+r^{\alpha_2})^{1/\alpha_1}} t^{-\alpha_1}dtds.
\end{align*}
Since the last term is equal to the first part of $I_2^{(1)}$, the desired upper bound follows.

Consider the integration over $A_{k'}^{1,2}(q,r)$. Again, by applying the induction hypotheses to $A_{k}(q+x_{k'1},r-x_{k'2})$, we get
\begin{align}
&\idotsint_{A^{1,2}_{k'}(q,r)}\prod_{i=1}^{k'}\frac{1}{|x_{i1}|^{\alpha_1}+|x_{i2}|^{\alpha_2}}\,dx_1 \dots dx_{k'}\nonumber\\
&\asymp\int_{1}^{r-1}\int_{1}^\infty\left(\frac{1}{(q+x_{k'1})^{\alpha_1}+(r-x_{k'2})^{\alpha_2}}\right)^{k-(k-1)(1/\alpha_1+1/\alpha_2)}
\frac{dx_{k'1}dx_{k'2}}{x_{k'1}^{\alpha_1}+x_{k'2}^{\alpha_2}}\label{case12}\\
&\lesssim\int_{1}^\infty(q+t)^{(k-1)(1+\alpha_1/\alpha_2)-k\alpha_1}\int_{1}^{r-1}(s+t^{\alpha_1/\alpha_2})^{-\alpha_2}dsdt\nonumber\\
&=\int_{1}^\infty(q+t)^{(k-1)(1+\alpha_1/\alpha_2)-k\alpha_1}\int_{1+t^{\alpha_1/\alpha_2}}^{r+t^{\alpha_1/\alpha_2}-1}s^{-\alpha_2}dsdt.\nonumber
\end{align}
Recall that the condition $k'-k(1/\alpha_1+1/\alpha_2)>0$ implies $\alpha_2>1$, hence the last integral is less than
\begin{align}
& C\int_{1}^\infty(q+t)^{(k-1)(1+\alpha_1/\alpha_2)-k\alpha_1}t^{\alpha_1/\alpha_2-\alpha_1}dt\nonumber\\
 &\lesssim  q^{(k-1)(1+\alpha_1/\alpha_2)-k\alpha_1}\int_{1}^q t^{\alpha_1/\alpha_2-\alpha_1}dt+\int_{q}^\infty t^{k(1+\alpha_1/\alpha_2)-k'\alpha_1-1}dt\nonumber\\
 & \lesssim q^{k(1+\alpha_1/\alpha_2)-k'\alpha_1}.\nonumber
\end{align}
It remains to estimate the integral \eqref{case12} in terms of $r$. We have
\begin{align*}
&\int_{1}^{r/2}\int_{1}^\infty\left(\frac{1}{(q+t)^{\alpha_1}+(r-s)^{\alpha_2}}\right)^{k-(k-1)(1/\alpha_1+1/\alpha_2)}\frac{dtds}{t^{\alpha_1}+s^{\alpha_2}}\\
&\lesssim  r^{(k-1)(1+\alpha_2/\alpha_1)-k\alpha_2}\int_{1}^{r/2}\int_{1}^\infty(t+s^{\alpha_2/\alpha_1})^{-\alpha_1}dtds\\
&\lesssim  r^{(k-1)(1+\alpha_2/\alpha_1)-k\alpha_2}\int_{1}^{r/2}s^{\alpha_2/\alpha_1-\alpha_2}ds.
\end{align*}
Since $\alpha_2<2$, we have $\alpha_2/\alpha_1-\alpha_2>-1$ and the last term above is less than $Cr^{k(1+\alpha_2/\alpha_1)-k'\alpha_2}$. 

Furthermore, for the second part of \eqref{case12} one gets
\begin{align*}
&\int_{r/2}^{r-1}\int_{1}^\infty\left(\frac{1}{(q+t)^{\alpha_1}+(r-s)^{\alpha_2}}\right)^{k-(k-1)(1/\alpha_1+1/\alpha_2)}\frac{dtds}{t^{\alpha_1}+s^{\alpha_2}}\\
&= \int_{1}^{r/2}\int_{1}^\infty\left(\frac{1}{(q+t)^{\alpha_1}+u^{\alpha_2}}\right)^{k-(k-1)(1/\alpha_1+1/\alpha_2)}\frac{dtdu}{t^{\alpha_1}+(r-u)^{\alpha_2}}\\
&\lesssim \int_{1}^{r/2}\int_{1}^\infty\left(\frac{1}{t^{\alpha_1}+u^{\alpha_2}}\right)^{k-(k-1)(1/\alpha_1+1/\alpha_2)}\frac{dtdu}{t^{\alpha_1}+r^{\alpha_2}}\\
&=\int_{1}^{r/2}\int_{1}^\infty\frac{(t^{\alpha_1}+u^{\alpha_2})^{(k-1)(1/\alpha_1+1/\alpha_2-1)}}{t^{\alpha_1}+u^{\alpha_2}}\frac{dtdu}{t^{\alpha_1}+r^{\alpha_2}}\\
&\leq \int_{1}^{r/2}\int_{1}^\infty\left(\frac{1}{t^{\alpha_1}+r^{\alpha_2}}\right)^{k-(k-1)(1/\alpha_1+1/\alpha_2)}\frac{dtdu}{t^{\alpha_1}+u^{\alpha_2}}\\
&\lesssim  r^{(k-1)(1+\alpha_2/\alpha_1)-k\alpha_2}\int_{1}^{r/2}\int_{1}^\infty(t+u^{\alpha_2/\alpha_1})^{-\alpha_1}dtdu\\
&\lesssim  r^{k(1+\alpha_2/\alpha_1)-k'\alpha_2},
\end{align*}
where the last inequality follows from estimating the first part of \eqref{case12}. Altogether we get that \eqref{case12} is less than
$$
C(q^{k(1+\alpha_1/\alpha_2)-k'\alpha_1}\wedge r^{k(1+\alpha_2/\alpha_1)-k'\alpha_2})\asymp (q^{\alpha_1}+r^{\alpha_2})^{k(1/\alpha_1+1/\alpha_2)-k'}.
$$
The case of $A_{k'}^{2,1}(q,r)$ is similar. Consider the integration over $A_{k'}^{2,4}(q,r)$. The induction hypothesis applied to $A_{k}(q-x_{k'1},x_{k'2}-r)$ gives
\begin{align*}
&\idotsint_{A^{2,4}_{k'}(q,r)}\prod_{i=1}^{k'}\frac{1}{|x_{i1}|^{\alpha_1}+|x_{i2}|^{\alpha_2}}\,dx_1 \dots dx_{k'}\nonumber\\
&\asymp\int_{r+1}^{\infty}\int_{1}^{q-1}\left(\frac{1}{(q-t)^{\alpha_1}+(s-r)^{\alpha_2}}\right)^{k-(k-1)(1/\alpha_1+1/\alpha_2)}\frac{dtds}{t^{\alpha_1}+s^{\alpha_2}}\\
&=\int_{1}^{\infty}\int_{1}^{q-1}\left(\frac{1}{(q-t)^{\alpha_1}+u^{\alpha_2}}\right)^{k-(k-1)(1/\alpha_1+1/\alpha_2)}\frac{dtdu}{t^{\alpha_1}+(u+r)^{\alpha_2}}\\
&=\int_{1}^{r-1}\int_{1}^{q-1}\ldots+\int_{r-1}^{\infty}\int_{1}^{q-1}\ldots
=\widehat I_1+\widehat I_2.
\end{align*}
In the last equality we have assumed $r\geq2$, the case $1\leq r<2$ being irrelevant. We have
\begin{align*}
\widehat I_1&\leq\int_{1}^{r-1}\int_{1}^{q-1}\left(\frac{1}{(q-t)^{\alpha_1}+u^{\alpha_2}}\right)^{k-(k-1)(1/\alpha_1+1/\alpha_2)}\frac{dtdu}{t^{\alpha_1}+(r-u)^{\alpha_2}}\\
&=\int_{1}^{r-1}\int_{1}^{q-1}\left(\frac{1}{(q-t)^{\alpha_1}+(r-v)^{\alpha_2}}\right)^{k-(k-1)(1/\alpha_1+1/\alpha_2)}\frac{dtdv}{t^{\alpha_1}+v^{\alpha_2}}.
\end{align*}
Note that the integral above is exactly the one from the case of $A_{k'}^{2,2}(q,r)$, so it is less than $C(q^{\alpha_1}+r^{\alpha_2})^{k(1/\alpha_1+1/\alpha_2)-k'}$. 
Furthermore,
\begin{align*}
\widehat I_2&=\int_{1}^{\infty}\int_{1}^{q-1}\left(\frac{1}{(q-t)^{\alpha_1}+(v+r-2)^{\alpha_2}}\right)^{k-(k-1)(1/\alpha_1+1/\alpha_2)}\frac{dtdv}{t^{\alpha_1}+(v+2r-2)^{\alpha_2}}\\
&\lesssim\int_{1}^{\infty}\int_{1}^{q-1}\left(\frac{1}{(q-t)^{\alpha_1}+(v+r)^{\alpha_2}}\right)^{k-(k-1)(1/\alpha_1+1/\alpha_2)}\frac{dtdv}{t^{\alpha_1}+v^{\alpha_2}}.
\end{align*}
The last integral above appears in the case of $A_{k'}^{2,1}(q,r)$, so it is less than $C(q^{\alpha_1}+r^{\alpha_2})^{k(1/\alpha_1+1/\alpha_2)-k'}$.
This gives the upper bound with the integration over $A_{k'}^{2,4}(q,r)$. The proof for $A_{k'}^{4,2}(q,r)$ is similar.

It remains to show the upper bound of \eqref{TRintegral} with the integration over the sets $A_{k'}^{1,3}(q,r)$, $A_{k'}^{3,1}(q,r)$, $A_{k'}^{2,3}(q,r)$, 
$A_{k'}^{3,2}(q,r)$, $A_{k'}^{3,4}(q,r)$ and $A_{k'}^{4,3}(q,r)$. This can be done in a similar way and is actually easier since the integration is reduced 
to a smaller number of variables. We omit the details.
\end{proof}

In order to evaluate the right hand side of \eqref{eq:DimTriple}, we prove the following equality which extends \cite[Lemma 1]{LX}.
\begin{lemma}\label{lem:series}
We have
\begin{align*}
&\inf\left\{\beta\in(0,2]:\sum_{m=1}^{\infty}\sum_{n=1}^{\infty}\frac{1}{m^{\beta}+n^{\beta}}
\left(\frac{1}{m^{\alpha_1}+ n^{\alpha_2}}\right)^{k-(k-1)(1/\alpha_1+1/\alpha_2)}<\infty\right\}\\
&=\max\left\{2-\alpha_1\left(k-(k-1)\left(\alpha^{-1}_1+\alpha^{-1}_2\right)\right), \
k\alpha_2\left(\alpha^{-1}_1+\alpha^{-1}_2-1\right)\right\}.
\end{align*}
\end{lemma}

\begin{proof}
The convergence of the series is equivalent to the convergence of the integral
\begin{align*}
&\int_1^{\infty}\int_1^{\infty}\left(\frac{1}{x}\wedge\frac{1}{y}\right)^{\beta}\left(\frac{1}{x^{\alpha_1}}
\wedge\frac{1}{y^{\alpha_2}}\right)^{k-(k-1)(1/\alpha_1+1/\alpha_2)}dxdy\\
&=\int_1^{\infty}\int_1^{y^{\alpha_2/\alpha_1}}\frac{1}{y^{\beta}}\left(\frac{1}{y^{\alpha_2}}\right)^{k-(k-1)(1/\alpha_1+1/\alpha_2)}dxdy\\
&\qquad +\int_1^{\infty}\int^y_{y^{\alpha_2/\alpha_1}}\frac{1}{y^{\beta}}\left(\frac{1}{x^{\alpha_1}}\right)^{k-(k-1)(1/\alpha_1+1/\alpha_2)}dxdy\\
&\qquad +\int_1^{\infty}\int_y^{\infty}\frac{1}{x^{\beta}}\left(\frac{1}{x^{\alpha_1}}\right)^{k-(k-1)(1/\alpha_1+1/\alpha_2)}dxdy\\
&:= I_1+I_2+I_3.
\end{align*}
It can be seen that $I_1<\infty$ if and only if
$$
\int_1^{\infty}y^{k\alpha_2(1/\alpha_1+1/\alpha_2-1)-\beta-1}dy<\infty,
$$
and the last condition is equivalent with $\beta>k\alpha_2(\alpha_1^{-1}+\alpha_2^{-1}-1)$.

Next we consider $I_2$. If $\alpha_1=\alpha_2$, then $I_2=0$, so assume that $\alpha_1\neq\alpha_2$. 
If $-k\alpha_1+(k-1)(1+\alpha_1/\alpha_2)=-1$,  then
$$
I_2= (1-\alpha_2/\alpha_1)\int_1^{\infty}y^{-\beta}\ln y \,dy.
$$
So $I_2<\infty$ if and only if $\beta>1=2-\alpha_1\left(k-(k-1)\left(\alpha^{-1}_1+\alpha^{-1}_2\right)\right)$. 
Suppose now $-k\alpha_1+(k-1)(1+\alpha_1/\alpha_2)\neq-1$. Then we have
$$
I_2=\int_1^{\infty}y^{-\beta}\left(\frac{y^{1-k\alpha_1+(k-1)(1+\alpha_1/\alpha_2)}-(y^{\alpha_2/\alpha_1})^{1-k\alpha_1
+(k-1)(1+\alpha_1/\alpha_2)}}{1-k\alpha_1+(k-1)(1+\alpha_1/\alpha_2)}\right)dy.
$$
We consider two cases:
\begin{enumerate}
\item[(a)] If $1-k\alpha_1+(k-1)(1+\alpha_1/\alpha_2)>0$, then
$$
I_2\leq C\int_1^{\infty}y^{1-k\alpha_1+(k-1)(1+\alpha_1/\alpha_2)-\beta}dy,
$$
and the last integral is finite if $\beta>2-\alpha_1\left(k-(k-1)\left(\alpha^{-1}_1+\alpha^{-1}_2\right)\right)$.
\item[(b)] If $1-k\alpha_1+(k-1)(1+\alpha_1/\alpha_2)<0$, then
$$
I_2\leq C\int_1^{\infty}y^{k\alpha_2\left(1/\alpha_1+1/\alpha_2-1\right)-1-\beta}dy,
$$
which  is finite if $\beta>k\alpha_2\left(\alpha^{-1}_1+\alpha^{-1}_2-1\right)$.
\end{enumerate}
Therefore, the condition $$\beta>\max\left\{2-\alpha_1\left(k-(k-1)\left(\alpha^{-1}_1+\alpha^{-1}_2\right)\right), \ 
k\alpha_2\left(\alpha^{-1}_1+\alpha^{-1}_2-1\right)\right\}$$ implies $I_2<\infty$.

Finally, we consider $I_3$. A necessary condition for $I_3<\infty$ is $$-k\alpha_1+(k-1)(1+\alpha_1/\alpha_2)-\beta<-1.$$ 
Assuming this we get
$$
I_3=\int_1^{\infty}\frac{y^{1-k\alpha_1+(k-1)(1+\alpha_1/\alpha_2)-\beta}}{1-k\alpha_1+(k-1)(1+\alpha_1/\alpha_2)-\beta}dy.
$$
Thus $I_3<\infty$ if and only if $\beta>2-\alpha_1\left(k-(k-1)\left(\alpha^{-1}_1+\alpha^{-1}_2\right)\right)$.

Therefore, we have proved that the condition
$$
\beta>\max\left\{2-\alpha_1\left(k-(k-1)\left(\alpha^{-1}_1+\alpha^{-1}_2\right)\right), \ k\alpha_2\left(\alpha^{-1}_1+\alpha^{-1}_2-1\right)\right\}
$$
implies $I_1,I_2,I_3<\infty$, and that the condition $I_1,I_3<\infty$ implies
$$
\beta>\max\left\{2-\alpha_1\left(k-(k-1)\left(\alpha^{-1}_1+\alpha^{-1}_2\right)\right), \ k\alpha_2\left(\alpha^{-1}_1+\alpha^{-1}_2-1\right)\right\}.
$$
This yields the conclusion of the lemma.
\end{proof}

\vspace{5mm}

\begin{proof1}
For $k=2$, the theorem is a reformulation of \cite[Corollary 3.8, Part (a)]{KMX}.
For $k\geq3$ and $\alpha_2<2$, the statement is a direct consequence of \eqref{eq:DimTriple}, \eqref{eq:2series}, 
Proposition~\ref{th:Ibeta} and Lemma~\ref{lem:series}. When $\alpha_2=\alpha_1=2$, we may apply Proposition~\ref{th:Ibeta} 
with $\widetilde\alpha_2:=\alpha_2-\varepsilon$ and let $\varepsilon\to0$ in order to obtain \tl{an} upper bound for $\beta$
in \eqref{eq:DimTriple}. However, this upper bound turns out to be 0, so the proof is complete.
\end{proof1}

\section{Existence of $k$-multiple points with $k\ge 3$}\label{sec4}

We will now focus on proving Theorem~\ref{th:ExMultPoints}. Since $X$ is symmetric, 
its \tl{continuous} transition density satisfies
$$
p_t(0)=\int_{\RR^d}\left(p_{t/2}(x)\right)^2dx>0.
$$
It follows from \cite[Proof of Theorem 1]{LRS} that the existence of $k$-multiple points of $X$ is 
equivalent to the existence of intersections of $k$ independent copies of $X$. Furthermore, by 
\cite[Theorem 2.1]{KX5}, $X$
is weakly unimodal. Hence, by \cite[Remark 6.6]{KX4}, $k$ independent
copies of $X$ intersect if and only if
\begin{equation}\label{eq:LevExpCond}
\int_{\RR^{d(k-1)}}\frac{1}{1+\Psi(\sum_{j=1}^{k-1}x_j)} \
\prod_{j=1}^{k-1}\frac{1}{1+\Psi(x_j)} \ d\overline{x}<\infty.
\end{equation}
We refer to \cite{E,LRS,FS,KX3} for appropriate conditions for the existence of intersections in terms of the potential density. 
Consider first $d=k=3$. Making the change of variables $x_1=\xi_1-\xi_{2}$, $x_2=\xi_{2}$ in \eqref{eq:LevExpCond}, we  
conclude that $M_3=\emptyset$  a.s. if and only if
\begin{equation}\label{eq:LevExpM3}
\int_{\RR^3}\int_{\RR^3}\frac{1}{1+\Psi(\xi_1)}\cdot\frac{1}{1+\Psi(\xi_1-\xi_2)}\cdot\frac{1}{1+\Psi(\xi_2)} \ d\xi_1d\xi_2<\infty.
\end{equation}
Since $\Psi(\xi)\leq\tl{C}\|\xi\|^2$ for all $\xi \in \RR^3$ with $\|\xi\|$ large enough, there exists $N\in\NN$ such that the 
integral in \eqref{eq:LevExpM3} can be estimated from below by
\begin{align*}
\tl{C\sum_{i=N}^\infty\frac{1}{i^2}\iint_{\|x\|,\|y\|\geq N}\frac{1}{\|x\|^2\|y\|^2} \ \mathds{1}_{\{i-1\leq\|x-y\|<i\}} \ dxdy.}
\end{align*}
Note that $\int \mathds{1}_{\{i-1\leq\tl{\|x-y\|}<i\}}dy\asymp i^2$, hence the last term above is greater than
\begin{align*}
\tl{C\sum_{i=N}^\infty\int_{\|x\|\geq N}\frac{1}{\|x\|^2}\cdot\frac{1}{\|x\|^2+i^2} \ dx\gtrsim
\sum_{i=N}^\infty\int_{\|x\|\geq i}\|x\|^{-4}dx\asymp\sum_{i=N}^\infty\frac{1}{i}=\infty.}
\end{align*}
This shows that $M_3=\emptyset$ a.s. for $d=3$ and the last statement of Theorem~\ref{th:ExMultPoints} follows. 
Furthermore, when $k=2$,  \eqref{eq:LevExpCond} is equivalent to
$$
\int_{\RR^d}\left(\frac{1}{1+\Psi(\xi)}\right)^2d\xi<\infty.
$$
Applying the same estimate of $\Psi$ as before we conclude that $M_2=\emptyset$ a.s. for $d\geq4$.

Recall \tl{from the Introduction} that for $d=2$ the stability exponent of
$X$ satisfies $B=PDP^{-1}$, where the matrix $D$ can have the following forms
\tl{\begin{enumerate}
\item[(A.1)] $\left(
	\begin{matrix}
		1/\alpha_1 & 0 \\
		0 & 1/\alpha_2  \\
	\end{matrix}
	\right) $ \ or \ $\left(
	\begin{matrix}
		1/\alpha_1 & -b \\
		b & 1/\alpha_2  \\
	\end{matrix}
	\right)\ \hbox{ with }\ \alpha_1 = \alpha_2;$
\item[(A.2)] $\left(
	\begin{matrix}
		1/\alpha & 0 \\
		1 & 1/\alpha  \\
	\end{matrix}
	\right)$.
\end{enumerate}}
\vspace{5mm}

\noindent We split the proof of Theorem~\ref{th:ExMultPoints} into two parts, according to the cases (A.1) and (A.2).\\

\begin{proof2a}
Note that, by Theorem~\ref{th:HdmMultPoints}, the Hausdorff dimension of $M_k$ is strictly positive for $\alpha_1
=\alpha_2=2$. Hence, without loss of generality we may assume $\alpha_2<2$. By \cite[(2.1),(2.5),(2.6),(2.7)]{KMX} we have
$$
\displaystyle\Psi(x)\asymp |x_1|^{\alpha_1}+|x_2|^{\alpha_2}\quad\text{ for }\|x\|\geq C.
$$
This implies that condition \eqref{eq:LevExpCond} is equivalent to
$$
\int_{\RR^{2(k-1)}}\frac{1}{1+|\sum_{j=1}^{k-1}x_{j1}|^{\alpha_1}+|\sum_{j=1}^{k-1}x_{j2}|^{\alpha_2}} \
\prod_{j=1}^{k-1}\frac{1}{1+|x_{j1}|^{\alpha_1}+|x_{j2}|^{\alpha_2}} \ d\overline{x}<\infty.
$$
The last integral is finite if and only if
$$
\sum_{m,n\in\NN}\frac{1}{m^{\alpha_1}+n^{\alpha_2}}\idotsint_{A_{k-1}(m,n)}\prod_{i=1}^{k-1}\frac{1}{|x_{i1}|^{\alpha_1}+|x_{i2}|^{\alpha_2}}\,dx_1 \dots dx_{k-1}<\infty,
$$
where the sets $A_{k-1}(m,n)$ were defined at the beginning of Section~\ref{sec3}.
Finally, by Proposition~\ref{th:Ibeta}, the last condition is equivalent to
$$
\sum_{m,n\in\NN}\left(\frac{1}{m^{\alpha_1}+n^{\alpha_2}}\right)^{k-(k-2)(1/\alpha_1+1/\alpha_2)}<\infty,
$$
and it is quite straightforward to see, using similar methods as in Lemma~\ref{lem:series}, that the last series is convergent  
if and only if $k-(k-1)(\alpha^{-1}_1+\alpha^{-1}_2)>0$.
\end{proof2a}

Observe that, since $\alpha_2\leq\alpha_1$, the second term in the dimension formula of Theorem~\ref{th:HdmMultPoints} satisfies
\begin{align*}
2-k\alpha_2\left(\alpha^{-1}_1+\alpha^{-1}_2-1\right)&=1-\frac{\alpha_2}{\alpha_1}+\alpha_2\left(k-(k-1)(\alpha^{-1}_1+\alpha^{-1}_2)\right)\\
&\geq \alpha_2\left(k-(k-1)(\alpha^{-1}_1+\alpha^{-1}_2)\right).
\end{align*}
Hence the condition $k-(k-1)(\alpha^{-1}_1+\alpha^{-1}_2)>0$ is equivalent to $\dimH M_k>0$ a.s. When the stability exponent of 
$X$ satisfies the case (A.1), it follows from Theorem~\ref{th:ExMultPoints} that
$$
M_k\neq\emptyset\text{ a.s.}\iff\dimH M_k>0 \text{ a.s.}
$$
However, as already mentioned in the Introduction, the last equivalence does not hold in the case (A.2), i.e., $X$ may have 
$k$-multiple points even if the Hausdorff dimension of $M_k$ is zero. An essential part of the proof is the following multiple integral estimate. 

\begin{proposition}\label{th:LogEst}
Fix $k\in\NN$ and assume $2(k-1)/k<\alpha<2$. Then for any real numbers $q,r\geq3$  we have	
\begin{equation}\label{eq:LogEst}
\idotsint_{A_{k}(q,r)}\prod_{i=1}^{k}\frac{dx_1 \dots dx_{k}}{(|x_{i1}|+|x_{i2}|\log\|x_i\|)^\alpha}\asymp\frac{1}{[q\vee(r\log r)]^{2-k(2-\alpha)}[\log(q\vee r)]^{k-1}}.
\end{equation}
\end{proposition}

\begin{proof}
As in the proof of Proposition \ref{th:Ibeta}, we restrict the integration in \eqref{eq:LogEst} to the set
\begin{align*}
\widetilde A_k(q,r):=&\biggl\{(x_1,...,x_k)\in\RR^{2k}:x_{k1}, x_{k2}\geq1, |x_{i1}|, |x_{i2}|\geq 1,i=1,...,k-1,\Bigl.\\
&\qquad \Bigl.q-1\leq \Big|\sum_{i=1}^kx_{i1}\Big|<q, \ r-1\leq \Big|\sum_{i=1}^kx_{i2}\Big|<r\biggl\},
\end{align*}
and we employ the sets $A_k^{i,j}(q,r)$ for $i,j=1,2,3,4$. Recall that
$$
A_k^{1,1}(q,r)\subseteq\widetilde A_k(q,r)\subseteq \bigcup_{i,j=1}^4 A_k^{i,j}(q,r),
$$
so it is enough to show the lower bound of \eqref{eq:LogEst} with the integration restricted to $A_k^{1,1}(q,r)$ and the upper bound 
of \eqref{eq:LogEst} with the integration over all $A_k^{i,j}(q,r)$. We proceed by induction on $k$. The case $k=1$ is obvious, as
$$
\iint_{A_{1}(q,r)}\frac{dx_1dx_2}{(|x_1|+|x_{2}|\log\|x\|)^\alpha}\asymp\left[q+r\log(q+r)\right]^{-\alpha}\asymp[q\vee(r\log r)]^{-\alpha}.
$$
Assume \eqref{eq:LogEst} holds for some $k\geq1$ and set $k':=k+1$. Consider first the integration over $A_{k'}^{1,1}(q,r)$. 
By the induction hypothesis applied to $A_{k}(q+x_{k'1},r+x_{k'2})$, we get
\begin{alignat*}{2}
&\idotsint_{A_{k'}^{1,1}(q,r)}\prod_{i=1}^{k'}\frac{dx_1 \dots dx_{k'}}{(|x_{i1}|+|x_{i2}|\log\|x_i\|)^\alpha}  \\
&\asymp \int_1^\infty\int_1^\infty\left(\frac{1}{q+x_{k'1}+(r+x_{k'2})\log(r+x_{k'2})}\right)^{2-k(2-\alpha)}\left(\frac{1}{\log(q+x_{k'1}+r+x_{k'2})}\right)^{k-1}\\
&\qquad \qquad \qquad \qquad \qquad \qquad \qquad \qquad \qquad \qquad \qquad \quad 
\times\frac{dx_{k'1}dx_{k'2}}{(x_{k'1}+x_{k'2}\log\|x_{k'}\|)^\alpha}.
\end{alignat*}
We split the right-hand side into four integrals
\begin{align*}
\int_1^\infty\int_1^\infty(\dots)&=\int_1^r\int_1^q(\dots)+\int_1^r\int_q^\infty(\dots)+\int_r^\infty\int_1^q(\dots)+\int_r^\infty\int_q^\infty(\dots)\\
&=I_1+I_2+I_3+I_4.
\end{align*}
Consider $I_4$ first. We have
\begin{align}\label{est:I4rq}
I_4\asymp\int_r^\infty\int_q^\infty\frac{dxdy}{(x+y\log y)^{2-k(2-\alpha)}[\log(x+y)]^{k-1}\left(x+y\log(x+y)\right)^\alpha}.
\end{align}
Assume $q<r\log r$. By Fubini,
\begin{align*}
I_4&\gtrsim\int_{r\log r}^\infty\int_r^x\frac{dydx}{(x+y\log y)^{2-k(2-\alpha)}(\log x)^{k-1}\left(x+y\log x\right)^\alpha}\\
&\geq\int_{r\log r}^\infty\int_r^x\frac{dydx}{(x+y\log x)^{k'\alpha-2k+2}(\log x)^{k-1}}\\
&=C\int_{r\log r}^\infty\frac{1}{(x+r\log x)^{k'\alpha-2k+1}(\log x)^{k}}\left[1-\left(\frac{x+r\log x}{x+x\log x}\right)^{k'\alpha-2k+1}\right]dx.
\end{align*}
Note that the condition $2(k'-1)/k'<\alpha$ implies $k'\alpha-2k>0$. Also, for $x\geq r\log r$ we have
$$
\frac{x+r\log x}{x+x\log x}\leq C<1.
$$
Hence
$$
I_4\gtrsim\int_{r\log r}^\infty\frac{dx}{(x+r\log x)^{k'\alpha-2k+1}(\log x)^{k}}.
$$
Furthermore, since $x^{-1}\log x$ is decreasing, for $x\geq r\log r$ we get
$$
\frac{r\log x}{x}\leq\frac{r\log(r\log r)}{r\log r}=\frac{\log r+\log(\log r)}{\log r}\leq2.
$$
Therefore
\begin{align*}
&\int_{r\log r}^\infty\frac{dx}{(x+r\log x)^{k'\alpha-2k+1}(\log x)^{k}}\asymp\int_{r\log r}^\infty\frac{dx}{x^{k'\alpha-2k+1}(\log x)^{k}}\\
&\asymp\frac{1}{(r\log r)^{k'\alpha-2k}[\log(r\log r)]^{k}}\asymp\frac{1}{(r\log r)^{2-k'(2-\alpha)}(\log r)^{k'-1}}.
\end{align*}
Let $q\geq r\log r$. \tl{A similar estimate} as above gives
\begin{align*}
I_4\gtrsim\int_{q}^\infty\int_r^x\frac{dydx}{(x+y\log x)^{k'\alpha-2k+2}(\log x)^{k-1}}&\asymp\int_{q}^\infty\frac{dx}{(x+r\log x)^{k'\alpha-2k+1}(\log x)^{k}}\\
&\asymp\frac{1}{q^{2-k'(2-\alpha)}(\log q)^{k'-1}}.
\end{align*}
Altogether,
$$
I_4\gtrsim\frac{1}{[q\vee(r\log r)]^{2-k'(2-\alpha)}[\log(q\vee r)]^{k'-1}}.
$$
This also proves the lower bound of \eqref{eq:LogEst}. It remains to show the upper bound. By \eqref{est:I4rq} we have
\begin{align*}
I_4&\lesssim\int_q^\infty\int_r^\infty\frac{dydx}{x^{2-k(2-\alpha)}(\log x)^{k-1}\left(x+y\log x\right)^\alpha}\\
&=C\int_q^\infty\frac{dx}{x^{2-k(2-\alpha)}(\log x)^{k}\left(x+r\log x\right)^{\alpha-1}}\\
&\lesssim\int_q^\infty\frac{dx}{x^{k'\alpha-2k+1}(\log x)^{k}}\asymp\frac{1}{q^{2-k'(2-\alpha)}(\log q)^{k'-1}}.
\end{align*}
On the other hand,
\begin{align*}
I_4&\lesssim\int_r^\infty\int_q^\infty\frac{dxdy}{(x+y\log y)^{k'\alpha-2k+2}(\log y)^{k-1}}
=C\int_r^\infty\frac{dy}{(q+y\log y)^{k'\alpha-2k+1}(\log y)^{k-1}}\\
&\lesssim\int_r^\infty\frac{dy}{y^{k'\alpha-2k+1}(\log y)^{k'\alpha-k}}\asymp\frac{1}{r^{k'\alpha-2k}(\log r)^{k'\alpha-k}}
=\frac{1}{(r\log r)^{2-k'(2-\alpha)}(\log r)^{k'-1}}.
\end{align*}
By taking the minimum of two upper bounds in terms of $q$ and $r$ we get
$$
I_4\lesssim\frac{1}{[q\vee(r\log r)]^{2-k'(2-\alpha)}[\log(q\vee r)]^{k'-1}}.
$$

Next we consider $I_1$. Notice that
\begin{align*}
I_1\asymp\frac{1}{[q\vee(r\log r)]^{2-k(2-\alpha)}[\log(q\vee r)]^{k-1}}\int_1^r\int_1^q\frac{dxdy}{\left[x+y\log(x+y)\right]^\alpha}.
\end{align*}
Since $\alpha>1$, we get by Fubini's theorem 
\begin{align*}
\int_1^r\int_1^q\frac{dxdy}{\left[x+y\log(x+y)\right]^\alpha}\lesssim\int_1^q\int_1^r\frac{dydx}{\left(x+y\log x\right)^\alpha}
\lesssim\int_1^q\frac{dx}{x^{\alpha-1}\log x}\lesssim\frac{q^{2-\alpha}}{\log q}.
\end{align*}
Hence $I_1\lesssim q^{k'(2-\alpha)-2}(\log q)^{1-k'}$. In a similar manner,
\begin{align*}
\int_1^r\int_1^q\frac{dxdy}{\left[x+y\log(x+y)\right]^\alpha}\lesssim\int_1^r\int_1^q\frac{dxdy}{\left(x+y\log y\right)^\alpha}
\lesssim\int_1^r\frac{dy}{(y\log y)^{\alpha-1}}\lesssim\frac{r^{2-\alpha}}{(\log r)^{\alpha-1}}.
\end{align*}
This gives $I_1\lesssim (r\log r)^{k'(2-\alpha)-2}(\log r)^{1-k'}$,
and the desired upper bound for $I_1$ follows. We have
\begin{align*}
I_2&\asymp\int_1^r\int_q^\infty\frac{dxdy}{[x+(r\log r)]^{2-k(2-\alpha)}[\log(x+r)]^{k-1}\left[x+y\log(x+y)\right]^\alpha}\\
&\lesssim\frac{1}{(r\log r)^{2-k(2-\alpha)}(\log r)^{k-1}}\int_1^r\int_q^\infty\frac{dxdy}{\left[x+y\log(2y)\right]^\alpha}\\
&\lesssim\frac{1}{(r\log r)^{2-k(2-\alpha)}(\log r)^{k-1}}\int_1^r\frac{dy}{\left(y\log(2y)\right)^{\alpha-1}}\\
&\lesssim\frac{1}{(r\log r)^{2-k(2-\alpha)}(\log r)^{k-1}}\cdot\frac{r^{2-\alpha}}{\left(\log r\right)^{\alpha-1}}
=\frac{1}{(r\log r)^{2-k'(2-\alpha)}(\log r)^{k'-1}}.
\end{align*}
Furthermore, by Fubini we obtain
\begin{align*}
I_2\leq\frac{1}{(\log q)^{k-1}}\int_q^\infty\int_1^r\frac{dydx}{x^{2-k(2-\alpha)}\left(x+y\log x\right)^\alpha}&\lesssim\frac{1}{(\log q)^{k-1}}\int_q^\infty\frac{dx}{x^{\alpha-k(2-\alpha)+1}\log x}\\
&\asymp\frac{1}{q^{2-k'(2-\alpha)}(\log q)^{k'-1}}.
\end{align*}
This gives the upper bound for $I_2$. Estimating $I_3$ is similar. Therefore we have proved the upper bound for the integral 
in \eqref{eq:LogEst}  over $A_{k'}^{1,1}(q,r)$.

Consider the integration over $A_{k'}^{2,2}(q,r)$. Note that we need to modify slightly the definition of the latter as the 
argument of the logarithm on the left-hand side of \eqref{eq:LogEst} does not approach 1. By the induction hypothesis applied to $A_{k}(q-x_{k'1},r-x_{k'2})$ we have
\begin{alignat*}{2}
&\idotsint_{A_{k'}^{2,2}(q,r)}\prod_{i=1}^{k'}\frac{dx_1 \dots dx_{k'}}{(|x_{i1}|+|x_{i2}|\log\|x_i\|)^\alpha}  \\
&\asymp \int_3^{r-3}\int_3^{q-3}\left(\frac{1}{q-x+(r-y)\log(r-y)}\right)^{2-k(2-\alpha)}&\left(\frac{1}{\log(q-x+r-y)}\right)^{k-1}\\
&&\times\frac{dxdy}{[x+y\log(x+y)]^\alpha}.
\end{alignat*}
As in the previous case, we split the right-hand side into
\begin{align*}
&\int_3^{r/2}\int_3^{q/2}(\dots)+\int_{r/2}^{r-3}\int_{q/2}^{q-3}(\dots)+\int_{3}^{r/2}\int_{q/2}^{q-3}(\dots)+\int_{r/2}^{r-3}\int_3^{q/2}(\dots)\\
&=J_1+J_2+J_3+J_4.
\end{align*}
We have
\begin{align*}
J_1\asymp\frac{1}{[q\vee(r\log r)]^{2-k(2-\alpha)}[\log(q\vee r)]^{k-1}}\int_3^{r/2}\int_3^{q/2}\frac{dxdy}{\left[x+y\log(x+y)\right]^\alpha}\asymp I_1,
\end{align*}
and the desired upper bound follows from the estimate of $I_1$. Furthermore,
\begin{align*}
J_2\asymp\frac{1}{[q\vee(r\log r)]^\alpha}\int_3^{r/2}\int_3^{q/2}\frac{dxdy}{(x+y\log y)^{2-k(2-\alpha)}[\log(x+y)]^{k-1}}.
\end{align*}
As $\alpha=2-k(2-\alpha)+(k-1)(2-\alpha)$, we get
\begin{align*}
J_2\lesssim\frac{1}{[q\vee(r\log r)]^{2-k(2-\alpha)}}\int_3^{r/2}\int_3^{q/2}\frac{dxdy}{(x+y\log y)^\alpha[\log(x+y)]^{k-1}}.
\end{align*}
Note that $x+y\log y\asymp x+y\log(x+y)$. Using similar arguments as for $I_1$ we get
\begin{align*}
\int_3^{r/2}\int_3^{q/2}\frac{dxdy}{[x+y\log(x+y)]^\alpha[\log(x+y)]^{k-1}}\lesssim \frac{q^{2-\alpha}}{(\log q)^k}\wedge\frac{r^{2-\alpha}}{(\log r)^{\alpha+k-2}}.
\end{align*}
This gives the upper bound for $J_2$. Since the cases of $J_3$ and $J_4$ are very similar to each other, we consider 
only the first one. For the latter one gets
\begin{align*}
J_3&\asymp\int_3^{r/2}\int_{q/2}^{q-3}\frac{dxdy}{(q-x+r\log r)^{2-k(2-\alpha)}[\log(q-x+r)]^{k-1}[x+y\log(x+y)]^\alpha}\\
&=\int_3^{r/2}\int_3^{q/2}\frac{dxdy}{(x+r\log r)^{2-k(2-\alpha)}[\log(x+r)]^{k-1}[q-x+y\log(q-x+y)]^\alpha}\\
&\asymp\int_3^{r/2}\int_3^{q/2}\frac{dxdy}{(x+r\log r)^{2-k(2-\alpha)}[\log(x+r)]^{k-1}[q+y\log(q+y)]^\alpha}.
\end{align*}
Hence we have
\begin{align*}
J_3&\lesssim\frac{1}{(r\log r)^{2-k(2-\alpha)}(\log r)^{k-1}}\int_3^{r/2}\int_3^{q/2}\frac{dxdy}{[x+y\log(x+y)]^\alpha},
\end{align*}
and applying \tl{the} previous estimate for $I_1$ we get the desired upper bound in terms of $r$. In order to estimate 
$J_3$ in terms of $q$ we can assume $q\geq r\log r$. We write
\begin{align*}
J_3&\asymp\int_3^{r/2}\int_3^{q/2}\frac{(x+r\log r)^{(k-1)(2-\alpha)}dxdy}{(x+r\log r)^{\alpha}[\log(x+r)]^{k-1}[q+y\log(q+y)]^\alpha}.
\end{align*}
As $q+y\log(q+y)\asymp q$ for $y\leq r$ we get
\begin{align*}
J_3&\lesssim q^{k(2-\alpha)-2}\int_3^{r/2}\int_3^{q/2}\frac{dxdy}{(x+r\log r)^{\alpha}(\log x)^{k-1}}\\
&\lesssim q^{k(2-\alpha)-2}\int_3^{r/2}\int_3^{q/2}\frac{dxdy}{[x+y\log(x+y)]^{\alpha}(\log x)^{k-1}},
\end{align*}
and since the last integral has already appeared in the case of $J_2$, the upper bound for $J_3$ is 
proved. This also completes the upper estimate with integration over $A_{k'}^{2,2}(q,r)$. The remaining 
integrals can be estimated in a similar way or reduced to the cases already considered by splitting the 
domain of integration in a suitable way. We omit the details.
\end{proof}

\vspace{5mm}

\begin{proof2b}
As in the proof for the case (A.1), we may assume without loss of generality that $\alpha<2$.
By \cite[(2.1),(2.5),(2.6),(2.7)]{KMX}, we have
$$
\displaystyle\Psi(x)\asymp |x_1|^{\alpha}+|x_2|^{\alpha}(\ln\|x\|)^{\alpha}\quad\text{ for }\|x\|\geq C,
$$
with $\alpha:=\alpha_1=\alpha_2=1/a$. In order to apply this estimate to \eqref{eq:LevExpCond} we split 
the domain of integration into $\|\sum_{j=1}^{k-1}x_{j}\|\leq2$ and $\|\sum_{j=1}^{k-1}x_{j}\|\geq2$. Note that the 
condition $\|\sum_{j=1}^{k-1}x_{j}\|\leq2$ implies
$$
|x_{(k-1)1}|\asymp\Big|\sum_{j=1}^{k-2}x_{j1}\Big|\quad\text{ and }\quad|x_{(k-1)2}|\asymp\Big|\sum_{j=1}^{k-2}
x_{j2}\Big|\quad\text{ for }\|x_{k-1}\|\ge C.
$$
This means that
$$
\int_{\RR^{d(k-1)}}\frac{1}{1+\Psi(\sum_{j=1}^{k-1}x_j)} \ \mathds{1}_{\{\|\sum_{j=1}^{k-1}x_{j}\|\leq2\}} \,
\prod_{j=1}^{k-1}\frac{1}{1+\Psi(x_j)} \ d\overline{x}<\infty
$$
if and only if
$$
\int_{\RR^{d(k-2)}}\frac{1}{1+\Psi(\sum_{j=1}^{k-2}x_j)} \
\prod_{j=1}^{k-2}\frac{1}{1+\Psi(x_j)} \ d\overline{x}<\infty.
$$
Since the last condition is equivalent to $M_{k-1}\neq\emptyset$, it is enough to consider the case
$\|\sum_{j=1}^{k-1}x_{j}\|\geq2$, and the final conclusion will follow by induction. The estimates of $\Psi$ 
imply that
$$
\int_{\RR^{d(k-1)}}\frac{1}{1+\Psi(\sum_{j=1}^{k-1}x_j)} \ \mathds{1}_{\{\|\sum_{j=1}^{k-1}x_{j}\|\geq2\}} \,
\prod_{j=1}^{k-1}\frac{1}{1+\Psi(x_j)} \ d\overline{x}<\infty
$$
if and only if
\begin{align*}
\int_{\RR^{2(k-1)}}&\frac{1}{1+|\sum_{j=1}^{k-1}x_{j1}|^{\alpha}+|\sum_{j=1}^{k-1}x_{j2}|^{\alpha}(\log\|\sum_{j=1}^{k-1}x_{j}\|)^{\alpha}} \\
&\times\mathds{1}_{\{\|\sum_{j=1}^{k-1}x_{j}\|\geq2\}} \ \prod_{j=1}^{k-1}\frac{1}{1+|x_{j1}|^{\alpha}+|x_{j2}|^{\alpha}(\log\|x_{j}\|)^\alpha} 
\, d\overline{x}<\infty.
\end{align*}
Employing the sets $A_{k-1}(m,n)$ as in the proof for the case (A.1), we conclude that the last condition is equivalent to
\begin{equation}\label{eq:SeriesLogInt}
\sum_{m,n\geq3}\left(\frac{1}{m+n\log(m+n)}\right)^\alpha\idotsint_{A_{k-1}(m,n)}\prod_{i=1}^{k-1}\frac{dx_1 \dots dx_{k-1}}{(|x_{i1}|+|x_{i2}|\log\|x_i\|)^\alpha}<\infty.
\end{equation}
Finally, Proposition ~\ref{th:LogEst} with $k-1$ instead of $k$ implies that \eqref{eq:SeriesLogInt} holds if and only if
\begin{equation}\label{ineq:LogSeries}
\sum_{m,n\geq3}\left(\frac{1}{m+n\log(m+n)}\right)^\alpha\frac{1}{[m\vee(n\log n)]^{2-(k-1)(2-\alpha)}[\log(m\vee n)]^{k-2}}<\infty.
\end{equation}
As already argued in the proof of Proposition~\ref{th:LogEst}, we have $m+n\log(m+n)\asymp m\vee (n\log n)$. Split the series 
in \eqref{ineq:LogSeries} into two parts
$$
\tl{\sum_{m,n\geq3}}(\dots)=\sum_{m\leq n\log n}(\dots)+\sum_{m> n\log n}(\dots)=S_1+S_2.
$$
We get
$$
S_1\asymp\sum_{m\leq n\log n}\frac{1}{n^{k\alpha-2k+4}(\log n)^{k\alpha-k+2}}\asymp\sum_{n\geq3}\frac{1}{n^{k\alpha-2k+3}(\log n)^{k\alpha-k+1}}.
$$
Thus $S_1<\infty$ if and only if $\alpha\geq 2(k-1)/k$. Furthermore,
$$
S_2\leq\sum_{m>n\log n}\frac{1}{m^{k\alpha-2k+4}(\log m)^{k-2}}.
$$
For $\alpha\geq 2(k-1)/k$ we have
\begin{align*}
\int_3^\infty\int_{x\log x}^\infty\frac{dydx}{y^{k\alpha-2k+4}(\log y)^{k-2}}&\asymp\int_3^\infty\frac{dx}{(x\log x)^{k\alpha-2k+3}[\log (x\log x)]^{k-2}}\\
&\asymp\int_3^\infty\frac{dx}{x^{k\alpha-2k+3}(\log x)^{k\alpha-k+1}}<\infty,
\end{align*}
and so $S_2<\infty$. Hence $S_1+S_2<\infty$ if and only if $\alpha\geq 2(k-1)/k$, as desired.
\end{proof2b}

\vspace{5mm}

We finish this section by characterizing the existence of double points of $X$. According to the Jordan 
decomposition for $d=3$, the stability exponent of $X$ satisfies $B=PDP^{-1}$, where the matrix $D$ may have 
one of the the following forms:\tl{
\begin{enumerate}
\item[(B.1)] $\left(
	\begin{matrix}
		1/\alpha_1 & 0 & 0 \\
		0 & 1/\alpha_2 & 0 \\
		0 & 0   & 1/\alpha_3 \\
	\end{matrix}
	\right) $, \ 
	$\left(
	\begin{matrix}
		1/\alpha_1 & -b & 0 \\
		b & 1/\alpha_2 & 0 \\
		0 & 0   & 1/\alpha_3 \\
	\end{matrix}
	\right)$ \ with  $\alpha_1= \alpha_2$ \\\\
\item[\ ] or 
	$\left(
	\begin{matrix}
		1/\alpha_1 & 0 & 0 \\
		0 & 1/\alpha_2 & -b \\
		0 & b   & 1/\alpha_3 \\
	\end{matrix}
	\right)$\, with  $\alpha_2= \alpha_3$;  \\\\
\item[(B.2)] $\left(
	\begin{matrix}
		1/\alpha_1 & 0 & 0 \\
		1 & 1/\alpha_2 & 0 \\
		0 & 0   & 1/\alpha_3 \\
	\end{matrix}
	\right)  \hbox{ with } \alpha_1= \alpha_2$ 
	or $\left(
	\begin{matrix}
		1/\alpha_1 & 0 & 0 \\
		0 & 1/\alpha_2 & 0 \\
		0 & 1   & 1/\alpha_3 \\
	\end{matrix}
	\right) \hbox{ with } \alpha_2= \alpha_3;$\\\\
	\item[(B.3)] $\left(
	\begin{matrix}
		1/\alpha & 0 & 0 \\
		1 & 1/\alpha & 0 \\
		0 & 1   & 1/\alpha \\
	\end{matrix}
	\right). $
\end{enumerate}
Clearly, in the case (B.3) we have $\alpha:=\alpha_1=\alpha_2=\alpha_3$.}
\begin{corollary}\label{cor:ExDouble}
Let $X = \{X(t), t \in \mathbb R_+\}$ be a symmetric operator semistable L\'evy process in $\RR^d$. A necessary and sufficient 
condition for the existence of double points of $X$ is the following:
\begin{itemize}
\item $d=2$: $M_2 \ne \emptyset$ \tl{almost surely}  if and only if $2-1/\alpha_1-1/\alpha_2>0$.
\item $d=3$, Cases (B.1) and (B.2): $M_2 \ne \emptyset$ \tl{almost surely} if and only if $\ 2-1/\alpha_1-1/\alpha_2-1/\alpha_3>0$.
\item $d=3$, Case (B.3): $M_2 \ne \emptyset$ \tl{almost surely} if and only if $\alpha\geq 3/2$.
\end{itemize}
Furthermore, $M_2=\emptyset$ almost surely for $d\geq4$.
\end{corollary}

\begin{proof}
The corollary is a reformulation of \cite[Theorem 5 and Theorem 7]{LX}, where the result was proved for symmetric operator 
stable L\'evy processes. The proof was based on \eqref{eq:LevExpCond} and on the asymptotics of the characteristic exponent 
derived in \cite[(4.9),(4.14),(4.15),(4.16)]{MX}. Analogous asymptotics for operator semistable L\'evy processes were obtained 
in \cite[(2.1),(2.5),(2.6),(2.7)]{KMX}. Therefore, the proof is the same as in \cite{LX}. The last statement is justified at the 
beginning of this section.
\end{proof}

\tl{
\begin{remark}
We note that the classification of the stability exponent $B$ and the corresponding asymptotics of the L\'evy exponent at 
infinity in \cite[Proofs of Theorems 5 and 7]{LX} was not quite correct. In dimension 2 and $B=PDP^{-1}$ with
$$
D=\left(
	\begin{matrix}
		1/\alpha_1 & -b \\
		b & 1/\alpha_2  \\
	\end{matrix}
	\right) \ \hbox{ and }\, \alpha_1=\alpha_2, 
$$
we remark that since the matrix is diagonalizable over the complex numbers, it follows 
from \cite[(2.1),(2.5),(2.6),(2.7)]{KMX} that
$$
\displaystyle\Psi(x)\asymp |x_1|^{\alpha_1}+|x_2|^{\alpha_2}\quad\text{ for }\|x\|\geq C.
$$
Similarly, in dimension 3 and $B=PDP^{-1}$ with
$$
D=\left(
	\begin{matrix}
		1/\alpha_1 & -b & 0 \\
		b & 1/\alpha_2 & 0 \\
		0 & 0   & 1/\alpha_3 \\
	\end{matrix}
	\right) \ \hbox{ and } \alpha_1=\alpha_2,
	$$
 or 
$$ D=\left(
	\begin{matrix}
		1/\alpha_1 & 0 & 0 \\
		0 & 1/\alpha_2 & -b \\
		0 & b   & 1/\alpha_3 \\
	\end{matrix}
	\right) \ \hbox{ and  } \alpha_2=\alpha_3,
$$
it follows from \cite[(2.1),(2.5),(2.6),(2.7)]{KMX} that
$$
\displaystyle\Psi(x)\asymp |x_1|^{\alpha_1}+|x_2|^{\alpha_2}+|x_3|^{\alpha_3}\quad\text{ for }\|x\|\geq C.
$$
This, however, does not change the statements of \cite[Theorems 5 and 7]{LX}.
\end{remark}
}

\tl{For completeness, we end this paper with the following remark on multiple points of a symmetric 
semistable L\'evy process in $\RR$.
\begin{remark} \label{Re:last}
The existence and Hausdorff dimension of multiple points for symmetric stable L\'evy processes in $\RR$ 
have been fully characterized in \cite{T}. If $X$ is a real-valued symmetric semistable L\'evy process with characteristic 
exponent $\Psi(\xi)$, the existence of multiple points and the Hausdorff dimension of $M_k$ can be explicitly determined 
by applying (\ref{eq:LevExpCond}) and (\ref{eq:HdmGeneral}), which only depend on the asymptotic behavior of 
$\Psi(\xi)$ at $\infty$. It follows from Choi \cite[Remark 2]{Ch} that the characteristic exponent of $X$ is of the form
$\Psi(\xi) = |\xi|^\alpha R(\xi)$, where $0 < \alpha < 2$ is a constant and $R(\xi)$ is a non-negative bounded and 
continuous function on $\RR$. Hence, if 
\begin{equation}\label{Eq:R}
\inf\{R(\xi): |\xi|\ge 1\} > 0,
\end{equation}
(by \cite[Remark 2]{Ch} this holds if the support of the L\'evy measure of 
$X$ spans $\RR$), then  $\Psi(\xi)\asymp|\xi|^\alpha$ for $\xi \in \RR$ with $|\xi|\ge 1.$ 
In view of the formula \eqref{eq:HdmGeneral} and the condition \eqref{eq:LevExpCond},  we can verify that under 
(\ref{Eq:R}) the characterizations of multiple points of $X$ are the same as those in \cite[Theorems 1 and 3]{T} for 
a symmetric $\alpha$-stable L\'evy process in $\RR$.
\end{remark}
}

\end{document}